\documentclass[orivec]{llncs}

\usepackage{amssymb,amsmath}
\usepackage[table]{xcolor}
\usepackage{graphicx}

\usepackage{url}
\usepackage{fca}
\usepackage{graphics}
\usepackage{alltt}
\usepackage{tikz}
\usepackage{float}
\usepackage{enumitem}
\usepackage{fontawesome}
\usepackage{multicol}

\newcommand\restr[2]{{
  \left.\kern-\nulldelimiterspace 
  #1 
  \vphantom{\big|} 
  \right|_{#2} 
  }}

\DeclareMathOperator{\ext}{ext}

\DeclareMathOperator{\dom}{dom}
\DeclareMathOperator{\df}{def}
\DeclareMathOperator{\rng}{rng}
\DeclareMathOperator{\schema}{schema}

\newcommand{\keywords}[1]{\par\addvspace\baselineskip
\noindent\keywordname\enspace\ignorespaces#1}

\usetikzlibrary{positioning}

\begin{document}
\mainmatter
\title{Orbital Semilattices}

\author{Jens K\"{o}tters and Stefan E. Schmidt}

\institute{Technische Universit\"{a}t Dresden}

\maketitle

\begin{abstract}
Orbital semilattices are introduced as bounded semilattices that are, in addition,
equipped with an outer multiplication (a semigroup action) and diagonals (a concept borrowed from cylindric algebra),
where each semilattice element has a certain domain.
An example of an orbital semilattice is a table algebra, where the domain is the
schema, the diagonals are diagonal relations, and outer multiplication encodes
renaming, projection and column duplication. It is shown that each orbital semilattice
can be represented by a subalgebra of such a table algebra.
\keywords{Relational Algebra, Cylindric Algebra, Lattice Theory}
\end{abstract}
\section{Motivation}
Concept lattices of relational structures~\cite{Ko13} have concepts with tables (as formalized in database theory~\cite{AHV95}) as their extents.
Orbital semilattices, introduced in this paper, are an attempt to axiomatize such concept lattices, but they do not capture completeness
(of the concept lattice). We show that orbital semilattices are, up to isomorphism, the subalgebras of certain table algebras. This is analog to
how cylindric algebras would be charactarized by cylindric set algebras~\cite{HMT71}.
\section{Preliminaries}
\subsection{Functions, Partial Transformations and Named Tuples}
A \emph{partial function} $f:A \rightharpoonup B$ is a relation $f \subseteq A \times B$ such that $(a,b_1) \in f$ and $(a,b_2) \in f$
implies $b_1=b_2$. Its \emph{domain of definition} is $\df(f):=\{a \in A \mid \exists b: (a,b) \in f\}$,
and its \emph{range} is $\rng(f):=\{b \in B \mid \exists a: (a,b) \in f\}$. For the same $f$, we may write $f:X \rightarrow B$
to indicate that $X = \df(f) \subseteq A$, and say that $f$ is a \emph{function} from $X$ to $B$. Similarly,
we may write $f:X \twoheadrightarrow Y$ if, in addition, $Y = \rng(f) \subseteq B$, and say that $f:X \twoheadrightarrow Y$ is a \emph{surjection}.
If $f$ is injective, we may say that $f:\df(f) \twoheadrightarrow \rng(f)$ is a \emph{bijection}. As usual, $f^{-1}$ denotes the inverse.
For an arbitrary $f:A \rightharpoonup B$, we write $f^{-1}(Z):=\{a \in \df(f) \mid f(a) \in Z\}$ to denote the \emph{preimage} of $Z$ under $f$.
The expression $\tfrac{z_1\dots z_n}{y_1\dots y_n}$ (where $y_1,\dots,y_n$ must be different from each other) is used to explicitly denote
a function $f:\{y_1,\dots,y_n\} \rightarrow \{z_1,\dots,z_n\}$ with $f(y_1)=z_1, \dots, f(y_n)=z_n$.

Throughout the paper, $\mathbf{var}=\{x_1,x_2,x_3,\dots\}$ denotes a countably infinite set of variables,
and $x_i$ always refers to the $i$-th variable in the given enumeration. Other variables, like $x,y,y_1,y_2,\dots$ are used
for arbitrary elements of $\mathbf{var}$.
An element of $\mathcal{T}_{\mathrm{fp}}(\mathbf{var}) := \{f:\mathbf{var} \rightharpoonup \mathbf{var} \mid \df(f) \text{ is finite}\}$
is called a \emph{finite partial transformation} on $\mathbf{var}$. A \emph{partial identity} is an element $\pi_X \in \mathbf{var}$
with $\pi_X:X \rightarrow \mathbf{var}$ and $\pi_X(x)=x$ for all $x \in X$. The composition $\mu \circ \lambda$ is defined for arbitrary
$\mu,\lambda \in \mathcal{T}_{\mathrm{fp}}(\mathbf{var})$, i.e. $\mu$ and $\lambda$ are composed as relations;
we do not require $\rng(\lambda) \subseteq \df(\mu)$. Accordingly, for arbitrary $\lambda \in \mathcal{T}_{\mathrm{fp}}(\mathbf{var})$
and $Z \subseteq \mathbf{var}$, we define the \emph{restriction} $\lambda|_Z:=\lambda \circ \pi_Z$
and the \emph{astriction} $\lambda|^{Z}:=\pi_Z \circ \lambda$. Using the astriction, an arbitrary composition $\mu \circ \lambda$ can
be represented as a composition of functions (in the usual sense), since $\mu \circ \lambda = \mu \circ \lambda|^{\df(\mu)}$
and $\lambda|^{\df(\mu)}:\lambda^{-1}(\df(\mu)) \rightarrow \df(\mu)$. Every $f \in \mathcal{T}_{\mathrm{fp}}(\mathbf{var})$
has a right inverse $f^{-r} \in \mathcal{T}_{\mathrm{fp}}(\mathbf{var})$ so that $f \circ f^{-r}=\pi_{\rng(f)}$.
We may e.g. set $f^{-r}(z)$ to the first variable $y$ in $x_1,x_2,x_3,\dots$ with $f(y)=z$.
An $f \in \mathcal{T}_{\mathrm{fp}}(\mathbf{var})$ is a \emph{folding} if $f \circ f = f$.
\begin{proposition}
\label{folding0}
An $f \in \mathcal{T}_{\mathrm{fp}}(\mathbf{var})$ is a folding if and only if $f \circ \pi_{\rng(f)} = \pi_{\rng(f)}$.
\end{proposition}
\begin{proof}
"$\Rightarrow$": If $f \circ f = f$, then $f \circ \pi_{\rng(f)} = f \circ f \circ f^{-r} = f \circ f^{-r} = \pi_{\rng(f)}$.
"$\Leftarrow$": We obtain $f \circ f = f \circ (\pi_{\rng(f)} \circ f) = (f \circ \pi_{\rng(f)}) \circ f = \pi_{\rng(f)} \circ f = f$.
\qed
\end{proof}

\begin{lemma}[Decomposition Lemma]
\label{decomp0}
Every finite partial transformation $f:A \rightarrow D$ can be decomposed into a folding $\delta:A \twoheadrightarrow B$,
a bijection $\sigma:B \twoheadrightarrow C$ and a partial identity $\pi:C \rightarrow D$,
so that $f = \pi \circ \sigma \circ \delta$.
\end{lemma}
\begin{proof}
The decomposition is $f = \pi_{\rng(f)} \circ (f^{-r})^{-1} \circ (f^{-r} \circ f)$.
\qed
\end{proof}
For a given set $G$, the set $\mathrm{NTup}(G) := \{t:\mathbf{var} \rightharpoonup G \mid \df(t) \text{ is finite}\}$
contains the \emph{named tuples} over $G$. The \emph{empty tuple} $\langle\rangle$ is the unique tuple
with $\df(\langle\rangle)=\emptyset$. We write $t \leq \tilde{t}$ if $\tilde{t} \circ \pi_{\df(t)} = t$ (i.e. if $\tilde{t}$ extends $t$),
and $t = t_1 \oplus t_2$ if $t_1,t_2 \leq t$ with $\rng(t) = \rng(t_1) \cup \rng(t_2)$.
%

\section{Orbital Semilattices}
\label{orbital1}
\subsection{Axiomatic Definition}
\begin{definition}[Orbital Semilattice]
An \emph{orbital semilattice} is an algebraic structure $(V,\wedge,0,1,\cdot,(d_{xy})_{x,y \in \mathbf{var}},\dom)$
consisting of an infimum operation $\wedge$, a \emph{bottom element} $0$, a \emph{top element} $1$,
a right multiplication $\cdot:V \times \mathcal{T}_{\mathrm{fp}}(\mathbf{var}) \rightarrow V$, a \emph{diagonal} $d_{xy} \in V$
for each $(x,y) \in \mathbf{var} \times \mathbf{var}$, and a \emph{domain function} $\dom:V \rightarrow \mathfrak{P}(\mathbf{var})$,
such that $(V,\wedge,0,1)$ is a bounded semilattice and the following axioms hold
for all $u,v \in V$, $\lambda,\mu \in \mathcal{T}_{\mathrm{fp}}(\mathbf{var})$, $x,y \in \mathbf{var}$
and $Y \in \mathfrak{P}_{\mathrm{fin}}(\mathbf{var})$.
\begin{multicols}{2}
\begin{enumerate}[label={\bf(A\arabic*)},leftmargin=1cm]
\item \label{axone0} $u \neq 0 \,\Rightarrow \, u \cdot \pi_{\emptyset} = 1$
\item \label{axzero0} $0 \cdot \lambda = 0$
\item \label{axdist0} $\dom(u) \subseteq Y \\\Rightarrow (u \wedge v) \cdot \pi_Y = u \wedge v \cdot \pi_Y$
\item \label{axproj0} $u \leq u \cdot \pi_Y$
\item \label{axmult0} $u \leq v \rightarrow u \cdot \lambda \leq v \cdot \lambda$
\item \label{axequal0} $(0 \neq u \leq \mathrm{d}_{xy} \text{ and } x \neq y) \\
    \Rightarrow u = (u \cdot \pi_{\dom(u) \setminus \{y\}}) \wedge d_{xy}$
\item \label{axsemi0} $(u \cdot \lambda) \cdot \mu = u \cdot (\lambda \circ \mu)$
\item \label{axneut0} $u \cdot \pi_{\dom(u)} = u$
\item \label{axdxx0} $d_{xx} \neq 0$
\item \label{axdxy0} $d_{xy} = d_{xx} \cdot \tfrac{xx}{xy}$
\item \label{axpreimg0} $u \neq 0 \\\Rightarrow \dom(u \cdot \lambda) = \lambda^{-1}(\dom(u))$
\item \label{axfinite0} $u \neq 0 \,\Rightarrow \, \dom(u) \mbox{ is finite}$
\item \label{axdom0} $\dom(u) = \{x \in \mathbf{var} \mid u \leq d_{xx}\}$
\end{enumerate}
\end{multicols}
\end{definition}
\subsection{Example: Table Algebra}
\label{tables0}
Let $G$ be a nonempty set. The elements of $\mathrm{Tab}(G):=\bigcup_{X \in \mathfrak{P}_{\mathrm{fin}}(\mathbf{var})} G^X$
are called \emph{tables} over $G$, and we set
\begin{align}
\label{schema0}
\schema(T) := \begin{cases}
X \quad \text{if } T \neq \emptyset \text{ and } T \subseteq G^X \\
\mathbf{var} \quad \text{if } T = \emptyset
\end{cases}
\quad.
\end{align}
\begin{definition}
\label{natjoin0}
The \emph{natural join} of $T_1,T_2 \in \mathrm{Tab}(G)$ is defined by
\begin{align*}
T_1 \Join T_2 := \{t \in G^{\schema(T_1) \cup \schema(T_2)} \mid t|_{\schema(T_1)} \in T_1 \text{ and } t|_{\schema(T_2)} \in T_2\} \;.
\end{align*}
\end{definition}
The natural join is associative, commutative and idempotent; in other words, $(\mathrm{Tab}(G),\Join)$ is an algebraic semilattice.
In the associated order $T_1 \leq T_2 :\Leftrightarrow T_1 \Join T_2 = T_1$, the natural join $\Join$ is the infimum.
For all $T \in \mathrm{Tab}(G)$ we have $\emptyset \Join T = \emptyset$ and $T \Join \{\langle\rangle\} = T$, which means
$\emptyset \leq T \leq \{\langle\rangle\}$ (so $\emptyset$ and $\{\langle\rangle\}$ are the bottom element and
top element, respectively, in $\mathrm{Tab}(G)$). The order is characterized in the following proposition.
\begin{proposition}
\label{order0}
Let $T_1,T_2 \in \mathrm{Tab}(G)$. Then
\begin{align}
T_1 \leq T_2 \;\Leftrightarrow\; \{t|_{\schema(T_2)} \mid t \in T_1\} \subseteq T_2 \quad.
\end{align}
\end{proposition}
\begin{proof}
"$\Rightarrow$": Let $t \in T_1 = T_1 \Join T_2$. Then by Def.~\ref{natjoin0} we have $t|_{\schema(T_2)} \in T_2$.
"$\Leftarrow$": Let $t \in T_1$. Trivially $t|_{\schema(T_1)} \in T_1$, and by assumption $t|_{\schema(T_2)} \in T_2$;
the latter implies $\schema(T_2) = \df(t|_{\schema(T_2)}) \subseteq \df(t) = \schema(T_1)$,
and thus $\df(t) = \schema(T_1) = \schema(T_1) \cup \schema(T_2)$, so we obtain $t \in T_1 \Join T_2$ from Def.~\ref{natjoin0}.
This shows $T_1 \leq T_1 \Join T_2$. But also $T_1 \Join T_2 \leq T_1$ (the natural join is the infimum),
so $T_1 = T_1 \Join T_2$, i.e. $T_1 \leq T_2$.
\qed
\end{proof}
The right multiplication $\cdot: \mathrm{Tab}(G) \times \mathcal{T}_{\mathrm{fp}}(\mathbf{var}) \rightarrow \mathrm{Tab}(G)$, defined by
\begin{align}
\label{rmult3}
T \cdot \lambda := \{ t \circ \lambda \mid t \in T \} \quad,
\end{align}
is a semigroup action, i.e. it satisfies $T \cdot \lambda \cdot \mu = T \cdot (\lambda \circ \mu)$
for all $T \in \mathrm{Tab}(G)$ and $\lambda,\mu \in \mathcal{T}_{\mathrm{fp}}(\mathbf{var})$.
For all $x,y \in \mathbf{var}$,
the \emph{diagonal} $E_{xy}$ is the table
\begin{align}
\label{exy0}
E_{xy} := \{t \in G^{\{x,y\}} \mid t(x)=t(y)\} \quad.
\end{align}
\begin{theorem}
\label{orbital0}
Let $G \neq \emptyset$. The 7-tuple $(\mathrm{Tab}(G),\Join,\emptyset,\{\langle\rangle\},\cdot,(E_{xy})_{x,y \in \mathrm{var}},\schema)$
is an orbital semilattice.
\end{theorem}
\begin{proof}
Some properites are immediate to see, and are stated for tables below:
\ref{axone0}: $T \neq \emptyset \Rightarrow T \cdot \pi_\emptyset = \{\langle\rangle\}$.
\ref{axzero0}: $\emptyset \cdot \lambda = \emptyset$.
\ref{axsemi0}: $T \cdot \lambda \cdot \mu = T \cdot (\lambda \circ \mu)$.
\ref{axneut0}: $T \cdot \pi_{\schema(T)} = T$.
\ref{axdxx0}: $E_{xx} = G^{\{x\}} \neq \emptyset$ since $G \neq \emptyset$.
\ref{axdxy0}: $E_{xy} = E_{xx} \cdot \frac{xx}{xy}$.
\ref{axpreimg0}: $T \neq \emptyset \Rightarrow \schema(T \cdot \lambda) = \lambda^{-1}(\schema(T))$.
\ref{axfinite0}: $T \neq \emptyset \Rightarrow \schema(T) \text{ is finite}$.
\ref{axdom0}: $\schema(T) = \{x \in \mathbf{var} \mid T \leq E_{xx}\}$.

It remains to show~\ref{axproj0}, \ref{axmult0}, \ref{axdist0} and \ref{axequal0}:

\ref{axproj0}: We obtain $T \leq T \cdot \pi_{\schema(T) \cap X} = T \cdot \pi_X$ from Prop.~\ref{order0} and~\eqref{rmult3}.

\ref{axmult0}: Assume $T_1 \leq T_2$. If $T_2 = \emptyset$, then also $T_1 = \emptyset$,
and $T_1 \cdot \lambda = \emptyset = T_2 \cdot \lambda$ by~\ref{axzero0}, so $T_1 \cdot \lambda \leq T_2 \cdot \lambda$ holds.
If $T_2 \neq \emptyset$, then $\schema(T_2 \cdot \lambda) = \lambda^{-1}(\schema(T_2))$ by~\ref{axpreimg0},
so for all $t \in T_1$ we have $t \circ \lambda \circ \pi_{\schema(T_2 \cdot \lambda)} = t \circ \lambda \circ \pi_{\lambda^{-1}(\schema(T_2))}
= t \circ \pi_{\schema(T_2)} \circ \lambda \in T_i \cdot \lambda$,
which shows $T_1 \cdot \lambda \leq T_2 \cdot \lambda$.

\ref{axdist0}: 
"$\leq$": We obtain $(T_1 \Join T_2) \cdot \pi_Y \leq T_2 \cdot \pi_Y$ from~\ref{axproj0},
and $(T_1 \Join T_2) \cdot \pi_Y \leq T_1 \cdot \pi_Y = T_1$ from~\ref{axproj0}, \eqref{rmult3} and the assumption $\schema(T_1) \subseteq Y$.
Taken together, this means $(T_1 \Join T_2) \cdot \pi_Y \leq T_1 \Join T_2 \cdot \pi_Y$.

"$\geq$": Let $t \in T_1 \Join T_2 \cdot \pi_Y$. Then $t = t_1 \oplus t_2|_Y$ for some $t_1 \in T_1$ and $t_2 \in T_2$,
which implies $t_1(x) = (t_2|_Y)(x)$ for all $x \in \df(t_1) \cap \df(t_2|_Y)$.
Since $\df(t_1) \subseteq Y$ by assumption, this means $t_1(x) = t_2(x)$ for all $x \in \df(t_1) \cap \df(t_2)$.
So there exists $\tilde{t}$ with $\tilde{t} = t_1 \oplus t_2 \in T_1 \Join T_2$, and thus $\tilde{t}|_Y \in (T_1 \Join T_2)|_Y$.
Next, we show $t_1 \oplus t_2|_Y = (t_1 \oplus t_2)|_Y$, which means $t=\tilde{t}|_Y$.

From the assumption $\df(t_1) \subseteq Y$, we obtain $\df(t_1 \oplus t_2|_Y) = \df((t_1 \oplus t_2)|_Y)$.
For $x \in \df(t_1)$, we have $(t_1 \oplus t_2|_Y)(x) = t_1(x) = ((t_1 \oplus t_2)|_Y)(x)$,
and for $x \in \df(t_2|_Y)$, we have $(t_1 \oplus t_2|_Y)(x) = t_2(x) = ((t_1 \oplus t_2)|_Y)(x)$.
We have thus shown $t_1 \oplus t_2|_Y = (t_1 \oplus t_2)|_Y$, or in short $t = \tilde{t}|_Y \in (T_1 \Join T_2) \cdot \pi_Y$.
This means $T_1 \Join T_2 \cdot \pi_Y \subseteq (T_1 \Join T_2) \cdot \pi_Y$,
and in particular $T_1 \Join T_2 \cdot \pi_Y \leq (T_1 \Join T_2) \cdot \pi_Y$.

\ref{axequal0}:
Assume $T \neq \emptyset$, $T \leq E_{xy}$ and $x \neq y$. Set $X:=\schema(T)\setminus\{y\}$.

"$\leq$":
Since $T \leq T \cdot \pi_X$ by~\ref{axproj0}, and $T \leq E_{xy}$,
we have $T \leq T \cdot \pi_X \Join E_{xy}$.

"$\geq$":
Let $s \in T \cdot \pi_X \Join E_{xy}$. Then there exist $t \in T$ and $e \in E_{xy}$ with
\begin{align}
\label{oplus1}
s = (t|_X) \oplus e \quad.
\end{align}
In particular $(t|_X)(x) = e(x)$ since $x \in \df(t|_X) \cap e$ (where $x \in X$ follows from $x \neq y$).
We obtain $s(y) = e(y) = e(x) = (t|_X)(x) = t(x) = t(y)$ from~\eqref{oplus1}, $e \in E_{xy}$ and $T \leq E_{xy}$;
moreover $s|_X = t|_X$ from~\eqref{oplus1}; so $s = t \in T$.
\qed
\end{proof}

\section{Derived Results}
\begin{proposition}
The following properties hold in every orbital lattice.
\label{props1}
\begin{multicols}{2}
\begin{enumerate}[label={\roman*)}]
\item $u \leq v \Rightarrow \dom(v) \subseteq \dom(u)$ \label{domain0}
\item $\dom(d_{xy}) = \{x,y\}$ \label{domdxy0}
\item $\dom(0) = \mathbf{var}$ \label{dom0}
\item $u \neq 0 \Leftrightarrow \dom(u) \text{ is finite}$ \label{finite0}
\item $u = 1 \Leftrightarrow \dom(u) = \emptyset$ \label{empty0}
\item $0 \neq 1$ \label{onezero0}
\item $1 \cdot \lambda = 1$ \label{one0}
\item $u \cdot \lambda = u \cdot \lambda|^{\dom(u)}$ \label{astr0}
\item $u \wedge v \neq 0 \\\Rightarrow \dom(u \wedge v) = \dom(u) \cup \dom(v)$ \label{dominf0}
\item $u \leq v \;\Leftrightarrow\; u \cdot \pi_{\dom(v)} \leq v$ \label{proj2}
\end{enumerate}
\end{multicols}
\end{proposition}
\begin{proof}
\ref{domain0}: This is immediate from~\ref{axdom0}.
\ref{domdxy0}: From $d_{xx} \leq d_{xx}$ and~\ref{axdom0} follows $x \in \dom(d_{xx})$, and we further obtain
$\dom(d_{xy}) = \dom(d_{xx} \cdot \tfrac{xx}{xy}) = (\tfrac{xx}{xy})^{-1}(\dom(d_{xx})) = \{x,y\}$ using~\ref{axdxy0},
\ref{axpreimg0}, \ref{axdxx0} and $x \in \dom(d_{xx})$.
\ref{dom0}: This is immediate from~\ref{axdom0}.
\ref{finite0}: This follows from~\ref{dom0} and~\ref{axfinite0}.
\ref{empty0}: Since $d_{xx} \neq 0$ by~\ref{axdxx0},
we obtain $\dom(1) = \dom(d_{xx} \cdot \pi_\emptyset)  = \pi_{\emptyset}^{-1}(\dom(d_{xx})) = \emptyset$ using~\ref{axone0} and~\ref{axpreimg0}.
Conversely, if $\dom(u)=\emptyset$, then $u \neq 0$ by~\ref{finite0}, and we obtain $u = u \cdot \pi_{\dom(u)} = u \cdot \pi_{\emptyset} = 1$
using~\ref{axneut0} and~\ref{axone0}.
\ref{onezero0}: This follows from~\ref{dom0} and~\ref{empty0}.
\ref{one0}: Considering $1 \neq 0$ by~\ref{onezero0}, we obtain $1 \cdot \lambda = 1 \cdot \pi_{\emptyset} \cdot \lambda = 1 \cdot \pi_{\emptyset} = 1$
from~\ref{axone0} and~\ref{axsemi0}.
\ref{astr0}: We obtain $u \cdot \lambda = u \cdot \pi_{\dom(u)} \cdot \lambda = u \cdot \lambda|^{\dom(u)}$ using~\ref{axneut0} and~\ref{axsemi0}.
\ref{dominf0}: We abbreviate $X:=\dom(u) \cup \dom(v)$. From $u \wedge v \leq u$ and~\ref{domain0} follows $\dom(u) \subseteq \dom(u \wedge v)$;
likewise, we obtain $\dom(v) \subseteq \dom(u \wedge v)$, and thus $X \subseteq \dom(u \wedge v)$.
Conversely, $(u \wedge v) \cdot \pi_X \leq u \cdot \pi_X = u \cdot \pi_{\dom(u)} = u$
by~\ref{axmult0}, \ref{astr0} and \ref{axneut0}, and likewise $(u \wedge v) \cdot \pi_X \leq v$,
thus $(u \wedge v) \cdot \pi_X \leq u \wedge v$. Since also $u \wedge v \leq (u \wedge v) \cdot \pi_X$ by~\ref{axproj0},
we obtain $u \wedge v = (u \wedge v) \cdot \pi_X$, so $\dom(u \wedge v) = \pi_X^{-1}(\dom(u \wedge v)) \subseteq X$
by $\ref{axpreimg0}$ and the assumption $u \wedge v \neq 0$. Altogether, we have shown $\dom(u \wedge v) = X$,
i.e. $\dom(u \wedge v) = \dom(u) \cup \dom(v)$.
\ref{proj2}: If $u \leq v$, then $u \cdot \pi_{\dom(v)} \leq v \cdot \pi_{\dom(v)} = v$ by~\ref{axmult0} and~\ref{axneut0}.
Conversely, if $u \cdot \pi_{\dom(v)} \leq v$, then $u \leq u \cdot \pi_{\dom(v)} \leq v$ by~\ref{axproj0} and the assumption.
\qed
\end{proof}

\begin{proposition}
\label{distr0}
Let $V$ be an orbital semilattice. Then
\begin{align}
\label{distr1}
(v_1 \wedge \dots \wedge v_n) \cdot \lambda = (v_1 \cdot \lambda) \wedge \dots \wedge (v_n \cdot \lambda)
\end{align}
for all $v_1,\dots,v_n \in V$ and one-to-one $\lambda \in \mathcal{T}_{\mathrm{fp}}(\mathbf{var})$
with $\dom(v_1) \cup \dots \cup \dom(v_n) \subseteq \rng(\lambda)$.
\end{proposition}
\begin{proof}
If $n=0$, then~\eqref{distr1} amounts to $1 \cdot \lambda = 1$, which holds by Prop.~\ref{props1}~\ref{one0}.
The case $n=1$ is trivial, so we may assume $n \geq 2$. By~\ref{axmult0} we have
$(v_1 \wedge \dots \wedge v_n) \cdot \lambda \leq v_i \cdot \lambda$ for all $i \in \{1,\dots,n\}$,
which already shows
\begin{align}
(v_1 \wedge \dots \wedge v_n) \cdot \lambda \leq (v_1 \cdot \lambda) \wedge \dots \wedge (v_n \cdot \lambda) \quad.
\end{align}
For the other inequality, we abbreviate $w:=(v_1 \cdot \lambda) \wedge \dots \wedge (v_n \cdot \lambda)$,
i.e. we have to show $w \leq (v_1 \wedge \dots \wedge v_n) \cdot \lambda$. This is obvious for $w=0$, so we may assume $w \neq 0$,
which implies $v_1,\dots,v_n \neq 0$ by~\ref{axzero0}. We then obtain
$\dom(w) = \dom(v_1 \cdot \lambda) \cup \dots \cup \dom(v_n \cdot \lambda)
= \lambda^{-1}(\dom(v_1)) \cup \dots \cup \lambda^{-1}(\dom(v_n))
= \lambda^{-1}(\dom(v_1) \cup \dots \cup \dom(v_n))
\subseteq \lambda^{-1}(\rng(\lambda))
= \df(\lambda)$ by Prop.~\ref{props1}~\ref{dominf0}, \ref{axpreimg0},
and the assumption $\dom(v_1) \cup \dots \cup \dom(v_n) \subseteq \rng(\lambda)$. From $\dom(w) \subseteq \df(\lambda)$, in turn,
we obtain
\begin{align}
\label{aux0}
w = w \cdot \pi_{\df(\lambda)}
\end{align}
using~\ref{axneut0} and Prop.~\ref{props1}~\ref{astr0}. For all $i \in \{1,\dots,n\}$, we have $w \leq v_i \cdot \lambda$ by the definition of $w$,
and thus $w \cdot \lambda^{-1} \leq v_i \cdot \lambda \cdot \lambda^{-1} = v_i \cdot \pi_{\rng(\lambda)} = v_i \cdot \pi_{\dom(v_i)} = v_i$
by~\ref{axmult0}, \ref{axsemi0}, Prop.~\ref{props1}~\ref{astr0} and~\ref{axneut0}.
This shows $w \cdot \lambda^{-1} \leq v_1 \wedge \dots \wedge v_n$, and consequently we obtain
$w = w \cdot \pi_{\df(\lambda)} = w \cdot \lambda^{-1} \cdot \lambda \leq (v_1 \wedge \dots \wedge v_n) \cdot \lambda$
from~\eqref{aux0}, \ref{axsemi0} and~\ref{axmult0}.
\qed
\end{proof}

\begin{proposition}
\label{equality0}
Let $y,z,y_1,y_2,z_1,z_2 \in \mathbf{var}$ with $y_1 \neq y_2$ and $z_1 \neq z_2$.
\begin{enumerate}[label={\roman*)},leftmargin=2cm]
\item $d_{yy} \cdot \tfrac{y}{z} = d_{zz}$ \label{equality1}
\item $d_{z_1z_2} \cdot \tfrac{z_1z_2}{y_1y_2}= d_{y_1y_2}$ \label{equality2}
\end{enumerate}
\end{proposition}
\begin{proof}
\ref{equality1}: Since $\dom(d_{yy} \cdot \tfrac{y}{z}) = (\tfrac{y}{z})^{-1}(\dom(d_{yy})) = \{z\}$
by~\ref{axpreimg0} and Prop.~\ref{props1}\,\ref{domdxy0},
we obtain $d_{yy} \cdot \tfrac{y}{z} \leq d_{zz}$ from~\ref{axdom0}. Symmetrically, we obtain $d_{zz} \cdot \tfrac{z}{y} \leq d_{yy}$,
and thus $d_{zz} = d_{zz} \cdot \pi_{\{z\}} = d_{zz} \cdot \tfrac{z}{y} \cdot \tfrac{y}{z} \leq d_{yy} \cdot \tfrac{y}{z}$ using~\ref{axneut0},
\ref{axsemi0} and~\ref{axmult0}. Taken together, this means $d_{yy} \cdot \tfrac{y}{z} = d_{zz}$.
\ref{equality2}: Using~\ref{axdxy0}, \ref{equality1} and~\ref{axsemi0}, we obtain
$d_{z_1z_2} \cdot \tfrac{z_1z_2}{y_1y_2}
= d_{z_1z_1} \cdot \tfrac{z_1z_1}{z_1z_2} \cdot \tfrac{z_1z_2}{y_1y_2}
= d_{y_1y_1} \cdot \tfrac{y_1}{z_1} \cdot \tfrac{z_1z_1}{z_1z_2} \cdot \tfrac{z_1z_2}{y_1y_2}
= d_{y_1y_1} \cdot \tfrac{y_1y_1}{y_1y_2}
= d_{y_1y_2}$.
\qed
\end{proof}

\begin{proposition}
\label{symmetry0}
Let $x,y \in \mathbf{var}$. Then $d_{xy} = d_{yx}$.
\end{proposition}
\begin{proof}
We obtain $d_{xy} = d_{xx} \cdot \frac{xx}{xy} = d_{yy} \cdot \frac{y}{x} \cdot \frac{xx}{xy} = d_{yy} \cdot \frac{yy}{yx} = d_{yx}$
from~\ref{axdxy0}, Prop.~\ref{equality0}~\ref{equality1} and~\ref{axsemi0}.
\qed
\end{proof}

\begin{definition}
\label{ediag0}
Let $\delta \in \mathcal{T}_{\mathrm{fp}}(\mathbf{var})$ be a folding. The \emph{$\delta$-diagonal} is
\begin{align}
e_\delta := \bigwedge_{x \in \df(\delta)} d_{x\delta(x)} \quad.
\end{align}
\end{definition}

\begin{proposition}
\label{duplication0}
Let $v \in V$, and let $\delta \in \mathcal{T}_{\mathrm{fp}}(\mathbf{var})$ be a folding with $\rng(\delta) \subseteq \dom(v)$.
Then $v \cdot \delta \leq e_{\delta}$.
\end{proposition}
\begin{proof}
Let $x \in \df(\delta)$. Since $\rng(\delta) \subseteq \dom(v)$, we have $\delta(x) \in \dom(v)$,
and thus $v \leq d_{\delta(x)\delta(x)}$ by~\ref{axdom0}.
Consequently, we obtain $v \cdot \delta \leq v \cdot \delta \cdot \pi_{\{\delta(x),x\}}
= v \cdot \tfrac{\delta(x)\delta(x)}{\delta(x)\;\;x} \leq d_{\delta(x)\delta(x)} \cdot \tfrac{\delta(x)\delta(x)}{\delta(x)\;\;x}
= d_{\delta(x)x} = d_{x\delta(x)}$
using~\ref{axproj0}, \ref{axsemi0}, \ref{axmult0}, \ref{axdxy0} and Prop.~\ref{symmetry0},
which shows $v \cdot \delta \leq e_\delta$ by Def.~\ref{ediag0}.
\qed
\end{proof}

\begin{proposition}
\label{duplication2}
Let $v \in V$, and let $\delta \in \mathcal{T}_{\mathrm{fp}}(\mathbf{var})$ be a folding with $\df(\delta) = \dom(v)$.
If $v \leq e_\delta$, then $v = (v \cdot \pi_{\rng(\delta)}) \wedge e_\delta$.
\end{proposition}
\begin{proof}
For $v = 0$, we use~\ref{axzero0} and obtain $0 = 0 \wedge e_\delta = (0 \cdot \pi_{\rng(\delta)}) \wedge e_\delta$.
For $v \neq 0$, we carry out an induction over $\#(\df(\delta)\setminus\rng(\delta))$.

\emph{Base Case: $\#(\df(\delta)\setminus\rng(\delta))=0$:} In this case $\df(\delta) = \rng(\delta)$,
so $\delta = \pi_{\rng(\delta)}$ since $\delta$ is a folding. By assumption $\df(\delta) = \dom(v)$,
so $\delta = \pi_{\dom(v)}$, and thus $v \leq e_{\pi_{\dom(v)}}$ by~\ref{axdom0}. So altogether
\begin{align}
v = v \wedge e_\delta
\underset{\ref{axneut0}}{=} (v \cdot \pi_{\dom(v)}) \wedge e_\delta
\underset{\text{Assmpt.}}{=} (v \cdot \pi_{\rng(\delta)}) \wedge e_\delta \quad.
\end{align}

\emph{Induction Step: $\#(\df(\delta)\setminus\rng(\delta))=n+1$:}
Let $y \in \df(\delta) \setminus \rng(\delta)$ and set $X:=\dom(v)\setminus\{y\}$.
Since $\delta$ is a folding, we have $y \neq \delta(y)$; and by assumption $v \leq e_\delta \leq d_{y\delta(y)}$;
so by~\ref{axequal0} we have
\begin{align}
\label{step0}
v = (v \cdot \pi_X) \wedge d_{y\delta(y)} \quad.
\end{align}
The restriction $\delta|_X:\df(\delta)\setminus\{y\} \twoheadrightarrow \rng(\delta)$ is a folding with
$\#(\df(\delta|_X) \setminus \rng(\delta|_X)) = n$. For all $x \in X$, we have $v \leq e_\delta \leq d_{x\delta(x)}$ by assumption,
and thus $v \cdot \pi_X \leq d_{x\delta(x)} \cdot \pi_X = d_{x\delta(x)} \cdot \pi_{\{x,\delta(x)\}} = d_{x\delta(x)}$
by~\ref{axmult0}, Prop.~\ref{props1}~\ref{domdxy0}\&\ref{astr0} and~\ref{axneut0},
which shows $v \cdot \pi_X \leq e_{\delta|_X}$ by Def.~\ref{ediag0}. In addition, using~\ref{axpreimg0} (recall $v \neq 0$)
we obtain $\dom(v \cdot \pi_X) = \pi_X^{-1}(\dom(v)) = X = \df(\delta|_X)$, so by the induction hypothesis
$v \cdot \pi_X = ((v \cdot \pi_X) \cdot \pi_{\rng(\delta|_X)}) \wedge e_{\delta|_X}$, and thus
\begin{align*}
v \underset{\substack{\eqref{step0}\\\text{I.H.}}}{=} (v \cdot \pi_X \cdot \pi_{\rng(\delta|_X)}) \wedge e_{\delta|_X} \wedge d_{y\delta(y)}
\underset{\substack{\ref{axsemi0}\\\text{Def.}\,\ref{ediag0}}}{=} (v \cdot \pi_{\rng(\delta)}) \wedge e_\delta \quad,
\end{align*}
which was to be shown.
\qed
\end{proof}

\begin{proposition}
\label{duplication1}
Let $v \in V$, and let $\delta \in \mathcal{T}_{\mathrm{fp}}(\mathbf{var})$ be a folding with $\df(\delta) = \dom(v)$.
If $v \leq e_\delta$, then $v = v \cdot \delta$.
\end{proposition}
\begin{proof}
If $v = 0$, then $0 = 0 \cdot \delta$ by~\ref{axzero0}; so we can now assume $v \neq 0$.
Since $\delta$ is a folding with $\df(\delta) = \dom(v)$, we have $\rng(\delta) \subseteq \dom(v)$,
and thus $v \cdot \delta \leq e_\delta$ by Prop.~\ref{duplication0}.
Also, we obtain $\df(\delta) = \delta^{-1}(\dom(v)) = \dom(v \cdot \delta)$ from $\rng(\delta) \subseteq \dom(v)$ and~\ref{axpreimg0},
so we can apply Prop.~\ref{duplication2} to $v \cdot \delta$ and $v$ alike:
\begin{align}
v \cdot \delta
\underset{\text{Prop.}\,\ref{duplication2}}{=} (v \cdot \delta \cdot \pi_{\rng(\delta)}) \wedge e_\delta
\underset{\substack{\text{Prop.}\,\ref{folding0}\\\ref{axsemi0}}}{=} (v \cdot \pi_{\rng(\delta)}) \wedge e_\delta
\underset{\substack{\text{Prop.}\,\ref{duplication2}\\\text{Assmpt.}}}{=} v \quad.
\end{align}
\end{proof}

\section{Labelings}
\label{labelings0}
\subsection{Tuple Labelings}
\begin{definition}[Tuple Labeling]
A \emph{tuple labeling} of an orbital semilattice $V$ over a nonempty set $G$ is a function
$\alpha:\mathrm{NTup}(G) \rightarrow V$ which satisfies
\begin{enumerate}[label={\bf(L\arabic*)},leftmargin=2cm]
\item \label{lax1} $\dom(\alpha(t)) = \df(t)$
\item \label{lax2} $\alpha(t \circ \lambda) = \alpha(t) \cdot \lambda$
\item \label{lax3} $(\df(t) \subseteq \dom(v) \text{ and } \alpha(t) \leq v \cdot \pi_{\df(t)}) \\
    \Rightarrow \exists \tilde{t} \in G^{\dom(v)}: (\tilde{t}|_{\df(t)} = t \text{ and } \alpha(\tilde{t}) \leq v)$
\item \label{lax4} $\alpha(t) \leq d_{z_1z_2} \,\Rightarrow\, t(z_1)=t(z_2)$
\end{enumerate}
for all $t \in \mathrm{NTup}(G)$, $\lambda \in \mathcal{T}_{\mathrm{fp}}(\mathbf{var})$, $v \in V$ and $z_1,z_2 \in \mathbf{var}$.
\end{definition}

\begin{definition}
\label{extalpha0}
Let $\alpha:\mathrm{NTup}(G) \rightarrow V$ be a tuple labeling. The \emph{$\alpha$-extent} of each $u \in V$
is given by the function $\ext_\alpha:V \rightarrow \mathrm{Tab}(G)$ with
\begin{align}
\ext_\alpha(u) := \{ t \in \mathrm{NTup}(G) \mid \alpha(t) \leq u \text{ and } \dom(\alpha(t))=\dom(u) \} \quad.
\end{align}
\end{definition}

\begin{theorem}
\label{embedding0}
Let $\alpha:\mathrm{NTup}(G) \rightarrow V$ be a tuple labeling with $\rng(\alpha) = V \setminus \{0\}$.
Then $\ext_\alpha:V \rightarrow \mathrm{Tab}(G)$ is an injective homomorphism of orbital semilattices.
\end{theorem}
\begin{proof}
\emph{Injectivity:} Let $u,v \in V$ with $u \neq v$. Then WLOG $u \not\leq v$, which implies $u \neq 0$,
and we have $\rng(\alpha)=V \setminus \{0\}$, so there exists $t \in \mathrm{NTup}(G)$ with $\alpha(t)=u$.
By Def.~\ref{extalpha0} this means $t \in \ext_\alpha(u)$, whereas $\alpha(t) = u \not \leq v$
implies $t \not\in \ext_\alpha(v)$, hence $\ext_{\alpha}(u) \neq \ext_{\alpha}(v)$.
This shows that $\ext_\alpha$ is injective.

\emph{Bottom Element:} Suppose there exists $t \in \ext_\alpha(0)$, then $\alpha(t) \leq 0$ by Def.~\ref{extalpha0},
so necessarily $\alpha(t)=0$, but this contradicts $\rng(\alpha)=V\setminus\{0\}$. This shows
\begin{align}
\label{bottom0}
\ext_\alpha(0) = \emptyset \quad.
\end{align}

\emph{Domain:}
For all $t \in \ext_\alpha(u)$ we have $\df(t)=\dom(\alpha(t))=\dom(u)$ by~\ref{lax1} and Def.~\ref{extalpha0},
which means $\ext_\alpha(u) \subseteq G^{\dom(u)}$. Then either $\ext_\alpha(u) = \emptyset$ or $\schema(\ext_\alpha(u)) = \dom(u)$.
In the case $\ext_\alpha(u) = \emptyset$, we obtain $\ext_\alpha(u) = \ext_\alpha(0)$ from~\eqref{bottom0},
and thus $u=0$ because $\ext_\alpha$ is injective, so $\schema(\ext_\alpha(u))=\schema(\emptyset)=\mathbf{var}=\dom(0)$
by~\eqref{bottom0}, \eqref{schema0} and Prop.~\ref{props1}~\ref{dom0}. So for all $u \in V$ we have
\begin{align}
\label{dom1}
\schema(\ext_{\alpha}(u)) = \dom(u) \quad.
\end{align}

\emph{Top Element:}
Since $\schema(\ext_\alpha(1))=\dom(1)=\emptyset$ by the previous result and Prop.~\ref{props1}~\ref{empty0},
we conclude $\ext_\alpha(1) \neq \emptyset$ and $\ext_\alpha(1) \subseteq G^\emptyset$, so necessarily
\begin{align}
\ext_\alpha(1)=\{\langle\rangle\} \quad.
\end{align}

\emph{Infimum:}
If $u \wedge v \neq \emptyset$, then $\dom(u \wedge v) = \dom(u) \cup \dom(v)$ by Prop.~\ref{props1}~\ref{dominf0},
and thus $\schema(\ext_{\alpha}(u \wedge v)) = \dom(u) \cup \dom(v)$ by~\eqref{dom1};
otherwise $\ext_\alpha(u \wedge v) = \emptyset$ by~\eqref{bottom0};
so in any case $\ext_\alpha(u \wedge v) \subseteq G^{\dom(u) \cup \dom(v)}$. We also obtain
$\ext_\alpha(u) \Join \ext_\alpha(v) \subseteq G^{\dom(u) \cup \dom(v)}$ from Def.~\ref{natjoin0} and~\eqref{dom1}.
So in order to show $\ext_{\alpha}(u \wedge v) = \ext_\alpha(u) \Join \ext_\alpha(v)$,
we may assume $t \in G^{\dom(u) \cup \dom(v)}$, and obtain
\begin{align}
t \in \ext_\alpha(u \wedge v) &\underset{\text{Def.}\,\ref{extalpha0}}{\Leftrightarrow} \alpha(t) \leq u \wedge v \label{help0} \\
&\Leftrightarrow \alpha(t) \leq u \text{ and } \alpha(t) \leq v \\
&\underset{\text{Prop.}\,\ref{props1}\,\ref{proj2}}{\Leftrightarrow} \alpha(t) \cdot \pi_{\dom(u)} \leq u \text{ and }
    \alpha(t) \cdot \pi_{\dom(v)} \leq v \\
&\underset{\ref{lax2}}{\Leftrightarrow} \alpha(t|_{\dom(u)}) \leq u \text{ and } \alpha(t|_{\dom(v)}) \leq v \\
&\underset{\text{Def.}\,\ref{extalpha0}}{\Leftrightarrow} t|_{\dom(u)} \in \ext_\alpha(u) \text{ and } t|_{\dom(v)} \in \ext_\alpha(v) \\
&\underset{\text{Def.}\,\ref{natjoin0}}{\Leftrightarrow} t \in \ext_\alpha(u) \Join \ext_\alpha(v)
\end{align}
(noting, for the reverse direction in~\eqref{help0}, that $\alpha(t) \leq u \wedge v$ implies $u \wedge v \neq 0$,
and thus $\dom(u \wedge v) = \dom(u) \cup \dom(v)$). This shows
\begin{align}
\label{natjoin1}
\ext_\alpha(u \wedge v) = \ext_\alpha(u) \Join \ext_\alpha(v) \quad.
\end{align}

\emph{Right Multiplication:}
In the case $u = 0$ we have $\ext_{\alpha}(0 \cdot \lambda) = \ext_{\alpha}(0) = \emptyset = \emptyset \cdot \lambda = \ext_{\alpha}(0) \cdot \lambda$
by~\ref{axzero0}, \eqref{bottom0} and~\eqref{rmult3}, and we are done. So now let $u \neq 0$. If, for an arbitrary $\lambda \in \mathrm{NTup}(G)$,
the astriction $\lambda|^{\dom(u)}:\lambda^{-1}(\dom(u)) \rightarrow \dom(u)$ satisfies
$\ext_{\alpha}(u \cdot \lambda|^{\dom(u)}) = \ext_{\alpha}(u) \cdot \lambda|^{\dom(u)}$, then also
$\ext_{\alpha}(u \cdot \lambda) = \ext_{\alpha}(u \cdot \lambda|^{\dom(u)}) = \ext_{\alpha}(u) \cdot \lambda|^{\dom(u)} = \ext_{\alpha}(u) \cdot \lambda$
by Prop.~\ref{props1}~\ref{astr0} (which applies to tables by Thm.~\ref{orbital0}). So WLOG we can assume $\lambda:Z \rightarrow \dom(u)$
for some $Z \in \mathfrak{P}_{\mathrm{fin}}(\mathbf{var})$.

Let $t \in \ext_{\alpha}(u)$, i.e. $\df(t) = \dom(u)$ and $\alpha(t) \leq u$. Then $\df(t \circ \lambda) = Z = \dom(u \cdot \lambda)$
by~\ref{axpreimg0}, and $\alpha(t \circ \lambda) = \alpha(t) \cdot \lambda \leq u \cdot \lambda$ by~\ref{lax2} and~\ref{axmult0},
so $t \circ \lambda \in \ext_{\alpha}(u \cdot \lambda)$. This shows one inclusion,
\begin{align}
\label{inclusion0}
\ext_{\alpha}(u) \cdot \lambda \subseteq \ext_{\alpha}(u \cdot \lambda) \quad.
\end{align}
For the other inclusion, we consider a decomposition $\lambda = \pi_X \circ \xi \circ \delta$ by Lemma~\ref{decomp0},
i.e. $\pi_X:X \rightarrow \dom(u)$ is a partial identity (and thus $X \subseteq \dom(u)$),
$\xi:Y \twoheadrightarrow X$ is a bijection, and $\delta:Z \twoheadrightarrow Y$ is a folding (and thus $Y \subseteq Z$),
where $X,Y \in \mathfrak{P}_{\mathrm{fin}}(\mathbf{var})$.
\begin{enumerate}
\item \emph{Partial identity $\pi_X:X \rightarrow \dom(u)$:}
Let $t \in \ext_{\alpha}(u \cdot \pi_X)$, i.e. $\alpha(t) \leq u \cdot \pi_X$ and $\df(t) = \dom(u \cdot \pi_X) = X$ by~\ref{axpreimg0}.
By~\ref{lax3} there exists $\tilde{t} \in G^{\dom(u)}$ with $\alpha(\tilde{t}) \leq u$ and $\tilde{t}|_X = t$,
so $\tilde{t} \in \ext_{\alpha}(u)$ and thus $t = \tilde{t} \circ \pi_X \in \ext_{\alpha}(u) \cdot \pi_X$.
This shows $\ext_{\alpha}(u \cdot \pi_X) \subseteq \ext_{\alpha}(u) \cdot \pi_X$, so combined with~\eqref{inclusion0} we have
\begin{align}
\label{factor0}
\ext_{\alpha}(u \cdot \pi_X) = \ext_{\alpha}(u) \cdot \pi_X \quad.
\end{align}

\item \emph{Bijection $\xi:Y \twoheadrightarrow X$:} We abbreviate $v:=u \cdot \pi_X$ and note $\dom(v) = X$.
Let $t \in \ext_{\alpha}(v \cdot \xi)$, i.e. $\alpha(t) \leq v \cdot \xi$
and $\df(t) = \dom(v \cdot \xi) = Y$ by~\ref{axpreimg0}. Then $\df(t \circ \xi^{-1})=X$, and $\alpha(t \circ \xi^{-1}) = \alpha(t) \cdot \xi^{-1}
\leq v \cdot \xi \cdot \xi^{-1} = v \cdot \pi_X = v$ by~\ref{lax2}, \ref{axmult0}, \ref{axsemi0} and~\ref{axneut0},
which means $t \circ \xi^{-1} \in \ext_{\alpha}(v)$, and thus $t = (t \circ \xi^{-1}) \circ \xi \in \ext_{\alpha}(v) \cdot \xi$.
This shows $\ext_{\alpha}(v \cdot \xi) \subseteq \ext_{\alpha}(v) \cdot \xi$. Combining this with~\eqref{inclusion0} and resolving $v$, we obtain
\begin{align}
\label{factor1}
\ext_{\alpha}(u \cdot \pi_X \cdot \xi) = \ext_{\alpha}(u \cdot \pi_X) \cdot \xi
\end{align}

\item \emph{Folding $\delta:Z \twoheadrightarrow Y$:} We abbreviate $w:=u \cdot \pi_X \cdot \xi$ and note $\dom(w) = Y$.
Let $t \in \ext_{\alpha}(w \cdot \delta)$, i.e. $\alpha(t) \leq w \cdot \delta$
and $\df(t) = \dom(w \cdot \delta) = Z$ by~\ref{axpreimg0}. Then $\df(t \circ \pi_Y)=Y$, and $\alpha(t \circ \pi_Y) = \alpha(t) \cdot \pi_Y
\leq w \cdot \delta \cdot \pi_Y = w \cdot \pi_Y = w$ by~\ref{lax2}, \ref{axmult0}, Prop.~\ref{folding0}, \ref{axsemi0} and \ref{axneut0},
which means $t \circ \pi_Y \in \ext_{\alpha}(w)$. Also, from $\alpha(t) \leq w \cdot \delta$ follows $\alpha(t) \leq e_\delta$
by Prop.~\ref{duplication0}, so for all $z \in Z$ we have $\alpha(t) \leq d_{z\delta(z)}$ by Def.~\ref{ediag0}, and thus
$t(z)=t(\delta(z))$ by~\ref{lax4}, which means $t = t \circ \delta$; since $t = (t \circ \pi_Y) \circ \delta \in \ext_{\alpha}(w) \cdot \delta$.
This shows $\ext_{\alpha}(w \cdot \delta) \subseteq \ext_{\alpha}(w) \cdot \delta$. Combining this with~\eqref{inclusion0} and resolving $w$, we obtain
\begin{align}
\label{factor2}
\ext_{\alpha}(u \cdot \pi_X \cdot \xi \cdot \delta) = \ext_{\alpha}(u \cdot \pi_X \cdot \xi) \cdot \delta
\end{align}
\end{enumerate}
From~\eqref{factor2}, \eqref{factor1} and~\eqref{factor0} combined follows
$\ext_{\alpha}(u \cdot \pi_X \cdot \xi \cdot \delta) = \ext_{\alpha}(u) \cdot \pi_X \cdot \xi \cdot \delta$,
and applying $\lambda=\pi_X \circ \xi \circ \delta$ and~\ref{axsemi0} results in
\begin{align}
\label{mult0}
\ext_{\alpha}(u \cdot \lambda) = \ext_{\alpha}(u) \cdot \lambda \quad.
\end{align}

\emph{Diagonals:}
Let $x,y \in \mathbf{var}$. For all $t \in \ext_{\alpha}(d_{xx})$ we have $\df(t) = \dom(\alpha(t)) = \dom(d_{xx}) = \{x\}$
by~\ref{lax1}, Def.~\ref{extalpha0} and Prop.~\ref{props1}~\ref{domdxy0}, so $\ext_{\alpha}(d_{xx}) \subseteq G^{\{x\}}$;
and conversely, for all $t \in G^{\{x\}}$, we have $\alpha(t) \leq d_{xx}$ by~\ref{axdom0},
so in fact $\ext_{\alpha}(d_{xx}) = G^{\{x\}} = E_{xx}$ (cf.~\eqref{exy0}), and thus
\begin{align*}
\ext_{\alpha}(d_{xy}) \underset{\ref{axdxy0}}{=} \ext_{\alpha}(d_{xx} \cdot \tfrac{xx}{xy})
\underset{\eqref{mult0}}{=} \ext_{\alpha}(d_{xx}) \cdot \tfrac{xx}{xy}
= E_{xx} \cdot \tfrac{xx}{xy}
\underset{\eqref{exy0}}{=} E_{xy} \quad.
\end{align*}
\end{proof}

\subsection{Quasi-Labelings}
\begin{definition}[Quasi-Labeling]
A \emph{quasi-labeling} of an orbital semilattice $V$ over a nonempty set $G$ is a function
$\alpha:\mathrm{NTup}(G) \rightarrow V$ which satisfies \ref{lax1}, \ref{lax2} and~\ref{lax3}.
\end{definition}

\begin{theorem}
\label{labeling0}
Let $\alpha:\mathrm{NTup}(G) \rightarrow V$ be a quasi-labeling. Then
\begin{align}
\label{cong0}
g \sim h \; :\Leftrightarrow \; \alpha(\langle x_1\mathord{:}g, x_2\mathord{:}h \rangle) \leq d_{x_1x_2}
\end{align}
defines an equivalence relation on $G$ which satisfies
\begin{align}
\label{exchange0}
g_1 \sim h_1, \dots, g_n \sim h_n \; \Rightarrow \;
 \alpha(\langle y_1\mathord{:}g_1,\dots,y_n\mathord{:}g_n \rangle) = \alpha(\langle y_1\mathord{:}h_1,\dots,y_n\mathord{:}h_n \rangle)
\end{align}
for all $y_1,\dots,y_n \in \mathbf{var}$ and $g_1,\dots,g_n,h_1,\dots,h_n \in G$, and the induced function
\begin{align}
\label{quotient0}
\bar{\alpha}:\begin{cases}
\mathrm{NTup}(G/\mathord{\sim}) \rightarrow V \\
\langle y_1\mathord{:}[g_1],\dots,y_n\mathord{:}[g_n] \rangle \mapsto \alpha(\langle y_1\mathord{:}g_1,\dots,y_n\mathord{:}g_n \rangle)
\end{cases}
\end{align}
is a tuple labeling.
\end{theorem}
\begin{proof}
For all $g \in G$, we have $\alpha(\langle x_1:g\rangle) \leq d_{x_1x_1}$ by~\ref{lax1} and~\ref{axdom0}, and thus
\begin{multline*}
\alpha(\langle x_1:g,x_2:g \rangle)
= \alpha(\langle x_1:g \rangle \circ \tfrac{x_1x_1}{x_1x_2})
\underset{\ref{lax2}}{=} \alpha(\langle x_1:g \rangle) \cdot \tfrac{x_1x_1}{x_1x_2} \\
\underset{\ref{axmult0}}{\leq} d_{x_1x_1} \cdot \tfrac{x_1x_1}{x_1x_2}
\underset{\ref{axdxy0}}{=} d_{x_1x_2} \quad,
\end{multline*}
which shows that the relation $\mathord{\sim} \subseteq G \times G$ is reflexive.

Next we show that the variables $x_1,x_2$ in~\eqref{cong0} can be replaced by any other $z_1,z_2 \in \mathbf{var}$ with $z_1 \neq z_2$:
If $\alpha(\langle x_1:g,x_2:h\rangle) \leq d_{x_1x_2}$, then
\begin{align*}
\alpha(\langle z_1:g, z_2:h\rangle)
\underset{\ref{lax2}}{=} \alpha(\langle x_1:g,x_2:h\rangle) \cdot \tfrac{x_1x_2}{z_1z_2}
\underset{\ref{axmult0}}{\leq} d_{x_1x_2} \cdot \tfrac{x_1x_2}{z_1z_2}
&\underset{\text{Prop.}\,\ref{equality0}\,\ref{equality2}}{=} d_{z_1z_2} \quad,
\end{align*}
and the converse implication is shown accordingly, so we obtain
\begin{align}
\label{gendef0}
\begin{split}
\text{for all $g,h \in G$ and $z_1,z_2 \in \mathbf{var}$ with $z_1\neq z_2$\,: } \\
    g \sim h \,\Leftrightarrow\, \alpha(\langle z_1:g,z_2:h \rangle) \leq d_{z_1z_2}\quad.\quad
\end{split}
\end{align}
In particular, if $g \sim h$, then $\langle x_1:h,x_2:g \rangle \leq d_{x_2x_1} = d_{x_1x_2}$ by~\eqref{gendef0} and Prop.~\ref{symmetry0},
so $h \sim g$, which shows that $\sim$ is symmetric.

Next we show~\eqref{exchange0}.
Let $s,t \in \mathrm{NTup}(G)$ with $\df(s)=\df(t)=:X$ such that $s(x) \sim t(x)$ for all $x \in X$;
we have to show $\alpha(s)=\alpha(t)$. Let $Y \subseteq \mathbf{var}$ with $\#Y=\#X$ and $X \cap Y = \emptyset$.
Then there exists a bijection $\xi:Y \rightarrow X$, and the named tuple $\tilde{s}:X \cup Y \rightarrow G$, which extends $s$ by a
variable-disjoint copy of $t$ via
\begin{align}
\label{merge0}
\tilde{s} := \begin{cases}
s(x) \quad \text{ if } x \in X \\
t(\xi(x)) \quad \text{ if } x \in Y
\end{cases}
\quad,
\end{align}
is well-defined. Also, we define a folding $\delta:X \cup Y \rightarrow X$ with
\begin{align}
\label{folding1}
\delta(x) := \begin{cases}
x \quad \text{ if } x \in X \\
\xi(x) \quad \text{ if } x \in Y
\end{cases}
\end{align}
and a permutation $\tau:X \cup Y \rightarrow X \cup Y$ with
\begin{align}
\label{swap0}
\tau(x) := \begin{cases}
\xi^{-1}(x) \quad \text{ if } x \in X \\
\xi(x) \quad \text{ if } x \in Y
\end{cases}
\quad.
\end{align}
For all $x \in X$, we obtain $\tilde{s}(x)=\tilde{s}(\delta(x))$ from~\eqref{merge0},
which implies $\tilde{s}(x) \sim \tilde{s}(\delta(x))$ because $\sim$ is reflexive.
Likewise, for all $y \in Y$ we have $s(\xi(y)) \sim t(\xi(y))$ by assumption,
and thus $\tilde{s}(\delta(y)) = \tilde{s}(\xi(y)) = s(\xi(y)) \sim t(\xi(y)) = \tilde{s}(y)$ by~\eqref{folding1} and~\eqref{merge0},
which implies $\tilde{s}(y) \sim \tilde{s}(\delta(y))$ because $\sim$ is symmetric. Taken together,
\begin{align}
\label{cong1}
\forall z \in X \cup Y: \tilde{s}(z) \sim \tilde{s}(\delta(z)) \quad.
\end{align}
Consequently, for all $z \in X \cup Y$ we have
\begin{align}
\alpha(\tilde{s}) \underset{\ref{axproj0}}{\leq} \alpha(\tilde{s}) \cdot \pi_{\{z,\delta(z)\}}
\underset{\ref{lax2}}{=} \alpha(\langle z:\tilde{s}(z),\delta(z):\tilde{s}(\delta(z))\rangle)
\underset{\substack{\eqref{cong1}\\\eqref{gendef0}}}{\leq} d_{z\delta(z)}
\end{align}
and thus $\alpha(\tilde{s}) \leq e_\delta$ by Def.~\ref{ediag0}. Using~\ref{lax1} we obtain
$\dom(\alpha(\tilde{s})) = \df(\tilde{s}) = X \cup Y = \df(\delta)$, so Prop.~\ref{duplication1} provides
\begin{align}
\label{reconstruction2}
\alpha(\tilde{s}) = \alpha(\tilde{s}) \cdot \delta \quad,
\end{align}
and we finally obtain
\begin{align*}
\alpha(s)
&\underset{\eqref{merge0}}{=} \alpha(\tilde{s} \circ \pi_X)
\underset{\ref{lax2}}{=} \alpha(\tilde{s}) \cdot \pi_X
\underset{\eqref{reconstruction2}}{=} \alpha(\tilde{s}) \cdot \delta \cdot \pi_X
\underset{\substack{\eqref{folding1}\\\eqref{swap0}}}{=} \alpha(\tilde{s}) \cdot (\delta \circ \tau) \cdot \pi_X \\
&\underset{\ref{axsemi0}}{=} \alpha(\tilde{s}) \cdot \delta \cdot \tau \cdot \pi_X
\underset{\eqref{reconstruction2}}{=} \alpha(\tilde{s}) \cdot \tau \cdot \pi_X
\underset{\ref{lax2}}{=} \alpha(\tilde{s} \circ \tau \circ \pi_X)
\underset{\substack{\eqref{swap0}\\\eqref{merge0}}}{=} \alpha(t) \quad,
\end{align*}
which proves~\eqref{exchange0}.

If $g_1 \sim g_2$ and $g_2 \sim g_3$, then $\alpha(\langle x_1:g_1, x_2:g_3 \rangle) =  \alpha(\langle x_1:g_1,x_2:g_2\rangle) \leq d_{x_1x_2}$
by~\eqref{exchange0} and~\eqref{cong0}, hence $g_1 \sim g_3$ by~\eqref{cong0}, which shows that $\sim$ is transitive.
So $\sim$ is an equivalence relation, and we obtain from~\eqref{exchange0} that $\bar{\alpha}$ in~\eqref{quotient0}
is well-defined.

It remains to be shown that $\bar{\alpha}$ is a tuple labeling. The function $q:G \rightarrow G / \sim$ with $q(g) := [g]$
allows to express~\eqref{quotient0} as
\begin{align}
\label{abstract0}
\bar{\alpha}(q \circ t) = \alpha(t) \quad,
\end{align}
and every element of $\mathrm{NTup}(G / \sim)$ can be represented as $q \circ t$ for some $t \in \mathrm{NTup}(G)$.
Consequently, the properties~\ref{lax1}, \ref{lax2} and \ref{lax3}, which inherit from the respective properties of the quasi-labeling $\alpha$,
can be shown as follows:
\ref{lax1}: $\dom(\bar{\alpha}(q \circ t)) = \dom(\alpha(t)) = \df(t) = \df(q \circ t)$
\ref{lax2}: $\bar{\alpha}((q \circ t) \circ \lambda) = \bar{\alpha}(q \circ (t \circ \lambda)) = \alpha(t \circ \lambda) = \alpha(t) \cdot \lambda
= \bar{\alpha}(q \circ t) \cdot \lambda$
\ref{lax3}: If $\bar{\alpha}(q \circ t) = v \cdot \pi_{\df(q \circ t)}$, then $\alpha(t) = \bar{\alpha}(q \circ t) = v \cdot \pi_{\df(q \circ t)}
= v \cdot \pi_{\df(t)}$, so $\alpha(\tilde{t}) = v$ for some $\tilde{t} \in \mathrm{NTup}(G)$ with $\tilde{t} \circ \pi_{\df(t)}=t$,
and thus $q \circ \tilde{t} \in \mathrm{NTup}(G / \sim)$ satisfies $\bar{\alpha}(q \circ \tilde{t}) = \alpha(\tilde{t}) = v$ as well as
$(q \circ \tilde{t}) \circ \pi_{\df(q \circ t)} = q \circ (\tilde{t} \circ \pi_{\df(t)}) = q \circ t$.

\ref{lax4}: Let $z_1,z_2 \in \mathbf{var}$ with $z_1 \neq z_2$ (the case $z_1=z_2$ is trivial).
If $\bar{\alpha}(q \circ t) \leq d_{z_1z_2}$, i.e. $\alpha(t) \leq d_{z_1z_2}$,
then $\dom(d_{z_1z_2}) \subseteq \dom(\alpha(t)) = \df(t)$ by Prop.~\ref{props1}~\ref{domain0} and~\ref{lax1},
so $z_1,z_2 \in \df(t)$ by Prop.~\ref{props1}~\ref{domdxy0}, and we obtain
$\alpha(\langle z_1:t(z_1),z_2:t(z_2) \rangle) = \alpha(t) \cdot \pi_{\{z_1,z_2\}} \leq d_{z_1z_2}$ using~\ref{lax2} and Prop.~\ref{props1}~\ref{proj2},
which means $t(z_1) \sim t(z_2)$ by~\eqref{gendef0}, and thus $[t(z_1)] = [t(z_2)]$, i.e. $(q \circ t)(z_1) = (q \circ t)(z_2)$.
This concludes the proof that $\bar{\alpha}$ is a tuple labeling.
\qed
\end{proof}

\section{Representation Theorem}
\subsection{Construction of a Quasi-Labeling}
\label{construction1}
With each orbital semilattice $V$, we associate an algebraic signature $S^V$, where
\begin{align}
\label{sign0}
S^V_n := \{v \in V \mid \dom(v) = \{x_1,\dots,x_{n+1}\}\}
\end{align}
contains the $n$-ary function symbols of $S^V$. In particular,
$S^V_0 = \{v \in V \mid \dom(v) = \{x_1\}\}$ contains the constants. A \emph{ground term} over $S^V$ is a sequence $vt_1\dots t_n$
where $v \in S^V_n$ and $t_1,\dots,t_n$ are ground terms; for $n=0$ these are the constants.
The set of ground terms is denoted by $\mathrm{GT}(S^V)$.

A set $B \subseteq \mathrm{GT}(S^V)$ is \emph{subterm-closed} if $vt_1\dots t_n \in B$ implies $t_1,\dots,t_n \in B$.
The \emph{subterm-closure} $A^{\downarrow}$ of a set $A \subseteq \mathrm{GT}(S^V)$ is the smallest subterm-closed superset of $A$.
A tuple $b \in \mathrm{NTup}(\mathrm{GT}(S^V))$ is a \emph{base tuple} if $b$ is injective and $\rng(b)$ is subterm-closed,
and we set $\mathrm{BTup}(\mathrm{GT}(S^V)) := \{b \in \mathrm{NTup}(\mathrm{GT}(S^V)) \mid b \text{ is base tuple}\}$.
With each $t \in \mathrm{NTup}(\mathrm{GT}(S^V))$ we associate some base tuple $b_t$ with $\rng(b_t)=\rng(t)^{\downarrow}$,
which exists since $\rng(t)^{\downarrow}$ is finite, and define in sequence
\begin{align}
\eta_{vt_1\dots t_n} &:= \langle x_1\mathord{:}t_1,\dots,x_n\mathord{:}t_n,x_{n+1}\mathord{:}vt_1\dots t_n\rangle \quad,\\
\tilde{\kappa}(b) &:= \bigwedge_{vt_1\dots t_n \in \rng(b)} v \cdot (\eta_{vt_1\dots t_n}^{-r} \circ b) \quad,\label{defk0}\\
\tilde{\alpha}(t) &:= \tilde{\kappa}(b_t) \cdot (b_t^{-1} \circ t) \label{defa0}
\end{align}
for all $vt_1\dots t_n \in \mathrm{GT}(S^V)$, $b \in \mathrm{BTup}(\mathrm{GT}(S^V))$
and $t \in \mathrm{NTup}(\mathrm{GT}(S^V))$, noting that the infimum in~\eqref{defk0} is finite.
Moreover, we inductively define $H^{(0)} := \emptyset$ and
\begin{multline}
\label{hdef0}
H^{(k+1)} := H^{(k)} \cup \{vt_1\dots t_n \in \mathrm{GT}(S^V) \mid
t_1,\dots,t_n \in H^{(k)} \text{ pairwise distinct}, \\
v \cdot \pi_{\{x_1,\dots,x_n\}} = \tilde{\alpha}(\langle x_1\mathord{:}t_1,\dots,x_n\mathord{:}t_n \rangle)
\}
\end{multline}
for $k \in \mathbb{N}$. Note that $H^{(1)}=S^V_0$ (the set of constants), since
$v \cdot \pi_{\emptyset} = \tilde{\alpha}(\langle\rangle)$ always holds
(we have $b_{\langle\rangle} = \langle\rangle$, and thus $\tilde{\alpha}(\langle\rangle) = \tilde{\kappa}(\langle\rangle) \cdot \pi_{\emptyset}
= 1 \cdot \pi_{\emptyset} = 1 = v \cdot \pi_{\emptyset}$ by \eqref{defa0}, \eqref{defk0} and \ref{axone0}).
Set $H := \bigcup_{k \in \mathbb{N}} H^{(k)}$. We will show that
$\alpha := \tilde{\alpha}|_{\mathrm{NTup}(H)}$ is a quasi-labeling. First, note that, by construction, $H$ has the following property:
\begin{proposition}
\label{hcon0}
A term $vt_1\dots t_n$ is contained in $H$ if and only if $t_1,\dots,t_n \in H$, $t_1,\dots,t_n$ pairwise distinct,
$\dom(v)=\{x_1,\dots,x_{n+1}\}$ and $v \cdot \pi_{\{x_1,\dots,x_n\}} = \alpha(\langle x_1\mathord{:}t_1,\dots,x_n\mathord{:}t_n\rangle)$.
\end{proposition}
Note that $H$ is subterm-closed by Prop.~\ref{hcon0}. So $t \in \mathrm{NTup}(H)$ implies $\rng(t)^{\downarrow} \subseteq H$,
and thus $b_t \in \mathrm{BTup}(H) = \{b \in \mathrm{NTup}(H) \mid b \text{ is base tuple}\}$. Setting $\kappa := \tilde{\kappa}|_{\mathrm{BTup}(H)}$,
and noting that $\eta_{vt_1\dots t_n}$ is injective for $vt_1\dots t_n \in H$, we may write
\begin{align}
\alpha(t) &= \kappa(b_t) \cdot (b_t^{-1} \circ t) \quad, \label{defa1} \\
\kappa(b) &= \bigwedge_{vt_1\dots t_n \in \rng(b)} v \cdot (\eta_{vt_1\dots t_n}^{-1} \circ b) \label{defk1}
\end{align}
for all $t \in \mathrm{NTup}(H)$ and $b \in \mathrm{BTup}(H)$.
\begin{proposition}
\label{props2}
Let $b \in \mathrm{BTup}(\mathrm{GT}(S^V))$.
\begin{enumerate}[label={\roman*)},leftmargin=1cm]
\item if $\tilde{\kappa}(b) \neq 0$, then $\dom(\tilde{\kappa}(b)) = \df(b)$ \label{domain1}
\item $\tilde{\kappa}(b \circ \xi) = \tilde{\kappa}(b) \cdot \xi$ for all bijections
 $\xi:X \twoheadrightarrow \df(b)$ in $\mathcal{T}_{\mathrm{fp}}(\mathbf{var})$ \label{bij0}
\item if $b = b_1 \oplus b_2$, then $\tilde{\kappa}(b) = \tilde{\kappa}(b_1) \wedge \tilde{\kappa}(b_2)$ \label{oplus0}
\end{enumerate}
\end{proposition}
\begin{proof}
\textbf{\ref{domain1}} For all $vt_1\dots t_n \in \rng(b)$, using axiom~\ref{axpreimg0} and~\eqref{sign0}, we obtain
$\dom(v \cdot (\eta^{-r}_{vt_1\dots t_n} \circ b)) = (\eta^{-r}_{vt_1\dots t_n} \circ b)^{-1}(\{x_1,\dots,x_{n+1}\})$, and thus
\begin{align}
\label{sandwich0}
b^{-1}(vt_1\dots t_n) \in \dom(v \cdot (\eta^{-r}_{vt_1\dots t_n} \circ b)) \subseteq \df(b) \quad.
\end{align}
If $\tilde{\kappa}(b) \neq 0$,
then $\dom(\tilde{\kappa}(b)) = \bigcup_{vt_1\dots t_n \in \rng(b)} \dom(v \cdot (\eta^{-r}_{vt_1\dots t_n} \circ b))$
by~\eqref{defk0} and Prop.~\ref{props1}\,\ref{dominf0}, i.e.~\eqref{sandwich0} translates to
$b^{-1}(vt_1\dots v_n) \in \dom(\tilde{\kappa}(b)) \subseteq \df(b)$,
which means $\dom(\tilde{\kappa}(b)) = \df(b)$.

\textbf{\ref{bij0}}
Let $\xi:X \twoheadrightarrow \df(b)$ be a bijection in $\mathcal{T}_{\mathrm{fp}}(\mathbf{var})$.
For all $vt_1\dots t_n \in \rng(b)$, we obtain $\dom(v \cdot (\eta^{-r}_{vt_1\dots t_n} \circ b)) \subseteq \df(b) = \rng(\xi)$
from~\eqref{sandwich0}.
This allows to obtain
$\tilde{\kappa}(b) \cdot \xi = \bigwedge_{ws_1\dots s_m \in \rng(b)} w \cdot (\eta^{-r}_{ws_1\dots s_m} \circ b) \cdot \xi
= \bigwedge_{ws_1\dots s_m \in \rng(b)} w \cdot (\eta^{-r}_{ws_1\dots s_m} \circ b \circ \xi) = \tilde{\kappa}(b \circ \xi)$
using Prop.~\ref{distr0}, \ref{axsemi0} and $\rng(b)=\rng(b \circ \xi)$.

\textbf{\ref{oplus0}}
Using~\eqref{defk0} and the assumption $\rng(b)=\rng(b_1) \cup \rng(b_2)$, we obtain
$\tilde{\kappa}(b) = (\bigwedge_{vt_1\dots t_n \in \rng(b_1)} (\eta_{vt_1\dots t_n}^{-r} \circ b))
\wedge (\bigwedge_{vt_1\dots t_n \in \rng(b_2)} (\eta_{vt_1\dots t_n}^{-r} \circ b))$.
Since $b_1 \leq b$, we have $\eta^{-r}_{vt_1\dots t_n} \circ b = \eta^{-r}_{vt_1\dots t_n} \circ b_1$ for all $vt_1\dots t_n \in \rng(b_1)$,
and likewise for $b_2$; so altogether $\tilde{\kappa}(b) = \tilde{\kappa}(b_1) \wedge \tilde{\kappa}(b_2)$.
\qed
\end{proof}
\begin{proposition}
\label{welldef2}
$\tilde{\alpha}(t) = \tilde{\kappa}(b) \cdot (b^{-1} \circ t)$ for all base tuples $b$ with $\rng(b)=\rng(t)^{\downarrow}$.
\end{proposition}
\begin{proof}
Since $\rng(b)=\rng(t)^{\downarrow}=\rng(b_t)$, a bijection $\xi:\df(b_t) \twoheadrightarrow \df(b)$ is given by $\xi:=b^{-1} \circ b_t$.
Then $\tilde{\alpha}(t) = \tilde{\kappa}(b_t) \cdot (b_t^{-1} \circ t)
= \tilde{\kappa}(b) \cdot \xi \cdot (b_t^{-1} \circ t) = \tilde{\kappa}(b) \cdot (b^{-1} \circ t)$
by~\eqref{defa0}, Prop.~\ref{props2}\,\ref{bij0} and~\ref{axsemi0}.
\qed
\end{proof}
\begin{lemma}
\label{core0}
Let $b \in \mathrm{BTup}(H)$ and $vt_1\dots t_n \in H$ such that $\rng(b)=\{vt_1\dots t_n\}^{\downarrow}$
and $\eta_{vt_1\dots t_n} \leq b$. Set $a:=b|^{\{t_1,\dots,t_n\}^{\downarrow}}$. Then
\begin{enumerate}[label={\roman*)},leftmargin=4cm]
\item $\kappa(b) \cdot \pi_{\{x_1,\dots,x_{n+1}\}} = v$ \quad, \label{core1}
\item $\kappa(b) \cdot \pi_{\df(a)} = \kappa(a)$ \quad. \label{core2}
\end{enumerate}
\end{lemma}
\begin{proof}
We obtain $\kappa(b) = (\bigwedge_{ws_1\dots s_m \in \rng(a)} w \cdot (\eta^{-1}_{ws_1\dots s_m} \circ b))
\wedge (v \cdot (\eta^{-1}_{vt_1\dots t_n} \circ b))$ from~\eqref{defk1} using $\rng(b) = \rng(a) \cup \{vt_1\dots t_n\}$.
Since $a \leq b$, we have $\eta^{-1}_{ws_1\dots s_m} \circ b = \eta^{-1}_{ws_1\dots s_m} \circ a$ for all $ws_1\dots s_m \in \rng(a)$,
and since $\eta_{vt_1\dots t_n} \leq b$, we have $\eta^{-1}_{vt_1\dots t_n} \circ b = \pi_{\{x_1,\dots,x_{n+1}\}}$.
So altogether, $\kappa(b) = \kappa(a) \wedge v \cdot \pi_{\{x_1,\dots,x_{n+1}\}}$. Since $v \cdot \pi_{\{x_1,\dots,x_{n+1}\}} = v$
by~\eqref{sign0} and~\ref{axneut0}, we arrive at
\begin{align}
\label{core3}
\kappa(b) &= \kappa(a) \wedge v \quad.
\end{align}
We have $vt_1\dots t_n \in H$, so $v\cdot\pi_{\{x_1,\dots,x_n\}} = \alpha(\langle x_1\mathord{:}t_1,\dots,x_n\mathord{:}t_n\rangle)$ by Prop.~\ref{hcon0}.
Moreover, $\alpha(\langle x_1\mathord{:}t_1,\dots,x_n\mathord{:}t_n\rangle) = \kappa(a) \cdot (a^{-1} \cdot
\langle x_1\mathord{:}t_1,\dots,x_n\mathord{:}t_n\rangle)$ by Prop.~\ref{welldef2}, since $\rng(a)=\{t_1,\dots,t_n\}^{\downarrow}$.
From $\eta_{vt_1\dots t_n} \leq b$ follows $\langle x_1\mathord{:}t_1,\dots,x_n\mathord{:}t_n \rangle \leq a$, so altogether we have
\begin{align}
\label{core4}
v \cdot \pi_{\{x_1,\dots,x_n\}} &= \kappa(a) \cdot \pi_{\{x_1,\dots,x_n\}} \quad.
\end{align}
\textbf{\ref{core1}}
This is obtained as follows (third equality detailed below):
\begin{multline*}
\kappa(b) \cdot \pi_{\{x_1,\dots,x_{n+1}\}}
\underset{\eqref{core3}}{=} (\kappa(a) \wedge v) \cdot \pi_{\{x_1,\dots,x_{n+1}\}}
\underset{\ref{axdist0}}{=} \kappa(a) \cdot \pi_{\{x_1,\dots,x_{n+1}\}} \wedge v \\
= \kappa(a) \cdot \pi_{\{x_1,\dots,x_n\}} \wedge v
\underset{\eqref{core4}}{=} v \cdot \pi_{\{x_1,\dots,x_n\}} \wedge v
\underset{\ref{axproj0}}{=} v \quad.
\end{multline*}
In the third equality, we use~\ref{axzero0} if $\kappa(a)=0$; otherwise, $\dom(\kappa(a)) = \df(a)$ by Prop.~\ref{props2}\,\ref{domain1},
and since $\df(a)=\df(b)\setminus\{x_{n+1}\}$, we obtain $\kappa(a) \cdot \pi_{\{x_1,\dots,x_{n+1}\}} = \kappa(a) \cdot \pi_{\{x_1,\dots,x_n\}}$
from Prop.~\ref{props1}\,\ref{astr0}.

\textbf{\ref{core2}}
As in the proof of~\ref{core1}, we obtain
$\kappa(b) \cdot \pi_{\df(a)} = (\kappa(a) \wedge v) \cdot \pi_{\df(a)} = \kappa(a) \wedge v \cdot \pi_{\{x_1,\dots,x_n\}}
= \kappa(a) \wedge \kappa(a) \cdot \pi_{\{x_1,\dots,x_n\}} = \kappa(a)$.
\qed
\end{proof}
\begin{proposition}
\label{catch0}
Let $vt_1\dots t_n \in H$. Then $\alpha(\eta_{vt_1\dots t_n}) = v$.
\end{proposition}
\begin{proof}
Let $b \in \mathrm{BTup}(H)$ be an extension of $\eta_{vt_1\dots t_n}$ with $\rng(b)=\{vt_1\dots t_n\}^{\downarrow}$.
Then $\alpha(\eta_{vt_1\dots t_n}) = \kappa(b) \cdot (b^{-1} \circ \eta_{vt_1\dots t_n}) = \kappa(b) \cdot \pi_{\{x_1,\dots,x_{n+1}\}} = v$
by Prop.~\ref{welldef2} and Lemma~\ref{core0}\,\ref{core1}.
\qed
\end{proof}
\begin{proposition}
\label{reduce0}
Let $a,b \in \mathrm{BTup}(H)$ with $a \leq b$. Then $\kappa(b) \cdot \pi_{\df(a)} = \kappa(a)$.
\end{proposition}
\begin{proof}
\textit{Special Case:}
We assume $\rng(b)=\{vt_1\dots t_n\}^{\downarrow}$ for some $vt_1\dots t_n \in H$, and $a = b|^{\{t_1,\dots,t_n\}^{\downarrow}}$.
Since $t_1,\dots,t_n,vt_1\dots t_n$ are pairwise distinct (cf. Prop.~\ref{hcon0}), there exists a base tuple $\hat{b}$ with
$\rng(\hat{b})=\{vt_1\dots t_n\}^{\downarrow}$ such that $\hat{b}^{-1}(t_1)=x_1,\dots,\hat{b}^{-1}(t_n)=x_n$ and $\hat{b}^{-1}(vt_1\dots t_n)=x_{n+1}$,
or shortly $\eta_{vt_1\dots t_n} \leq \hat{b}$. Set $\hat{a} := \hat{b}|^{\{t_1,\dots,t_n\}^{\downarrow}}$.
Since $\rng(\hat{b})=\rng(b)$ and $\rng(\hat{a})=\rng(a)$, we obtain the bijections $\hat{b}^{-1} \circ b:\df(b) \rightarrow \df(\hat{b})$
and $\hat{a}^{-1} \circ a:\df(a) \rightarrow \df(\hat{a})$, and we have
$(\hat{b}^{-1} \circ b) \circ \pi_{\df(a)} = \hat{b}^{-1} \circ (b \circ \pi_{\df(a)}) = \hat{b}^{-1} \circ a = \hat{a}^{-1} \circ a$, hence also
\begin{align}
\label{switch0}
(\hat{b}^{-1} \circ b) \circ \pi_{\df(a)} = \pi_{\df(\hat{a})} \circ (\hat{a}^{-1} \circ a) \quad.
\end{align}
We now obtain $\kappa(b) \cdot \pi_{\df(a)} = \kappa(\hat{b}) \cdot (\hat{b}^{-1} \circ b) \cdot \pi_{\df(a)}
= \kappa(\hat{b}) \cdot \pi_{\df(a')} \cdot (\hat{a}^{-1} \circ a)
= \kappa(\hat{a}) \cdot (\hat{a}^{-1} \circ a) = \kappa(a)$ using Prop.~\ref{props2}\,\ref{bij0}, \eqref{switch0} and \ref{axsemi0},
and Lemma~\ref{core0}\,\ref{core2}.

\textit{Extended Case:}
We assume $a \lessdot b$, i.e. $a = b|^{\rng(b)\setminus\{vt_1\dots t_n\}}$ for some $vt_1\dots t_n \in \rng(b)$.
Setting $b^*:=b|^{\{vt_1\dots t_n\}^\downarrow}$ and $a^*:=a|^{\{t_1,\dots,t_n\}^\downarrow}$, we obtain
$b=a \oplus b^*$, $a^*=b^*|^{\{t_1,\dots,t_n\}^{\downarrow}}$, and trivially $a = a \oplus a^*$. With this, we obtain
$\kappa(b) \cdot \pi_{\df(a)} = (\kappa(a) \wedge \kappa(b^*)) \cdot \pi_{\df(a)} = \kappa(a) \wedge \kappa(b^*) \cdot \pi_{\df(a^*)}
= \kappa(a) \wedge \kappa(a^*) = \kappa(a)$ using Prop.~\ref{props2}\,\ref{oplus0}, \ref{axdist0} and Prop.~\ref{props1}\,\ref{astr0},
Special Case 1, and again Prop.~\ref{props2}\,\ref{oplus0}.

\textit{General Case:} Shown by induction over $\#(\rng(b)\setminus\rng(a))$.
\textit{Base Step:} We obtain $\kappa(b) \cdot \pi_{\df(b)} = \kappa(b)$ from Prop.~\ref{props2}\,\ref{domain1} and~\ref{axneut0}
(from~\ref{axzero0} if $\kappa(b)=0$). \textit{Induction Step:} Select $vt_1\dots t_n$ minimal in $\rng(b)\setminus\rng(a)$.
Then $\tilde{a}:=b|^{\rng(a)\cup\{vt_1\dots t_n\}}$ is a base tuple with $a \lessdot \tilde{a} \leq b$, and thus
$\kappa(b) \cdot \pi_{\df(a)} = \kappa(b) \cdot \pi_{\df(\tilde{a})} \cdot \pi_{\df(a)} = \kappa(\tilde{a}) \cdot \pi_{\df(a)} = \kappa(a)$
by~\ref{axsemi0}, induction hypothesis, and the Extended Case.
\qed
\end{proof}
\begin{proposition}
\label{welldef3}
If $\rng(t) \subseteq \rng(b)$ then $\alpha(t) = \kappa(b) \cdot (b^{-1} \circ t)$.
\end{proposition}
\begin{proof}
Note that $\rng(t)^{\downarrow} \subseteq \rng(b)$, since $b$ is subterm-closed.
So $a := b|^{\rng(t)^{\downarrow}}$ satisfies $a \leq b$ and $\rng(a)=\rng(t)^{\downarrow}$.
Then $\kappa(b) \cdot (b^{-1} \circ t) = \kappa(b) \cdot \pi_{\df(a)} \cdot (a^{-1} \circ t) = \kappa(a) \cdot (a^{-1} \circ t) = \alpha(t)$
by~\ref{axsemi0}, Prop.~\ref{reduce0} and Prop.~\ref{welldef2}.
\qed
\end{proof}
\begin{proposition}
\label{sane0}
$\kappa(b) \neq 0$.
\end{proposition}
\begin{proof}
Suppose $\kappa(b)=0$. Then $1 = \kappa(\langle\rangle) = \kappa(b) \cdot \pi_{\df(\emptyset)} = 0$ by Prop.~\ref{reduce0} and~\ref{axzero0},
which contradicts Prop.~\ref{props1}\,\ref{onezero0}.
\qed
\end{proof}
\begin{proposition}
\label{lplus0}
$\alpha:\mathrm{NTup}(H) \rightarrow V$ satisfies \ref{lax1}, \ref{lax2} and
\begin{align*}
\ref{lax3}^{+}\quad\alpha(t)=v \cdot \pi_{\df(t)} \Rightarrow
 \exists \tilde{t} \in \mathrm{NTup}(H): t \leq \tilde{t} \text{ and } \alpha(\tilde{t}) = v \quad.
\end{align*}
\end{proposition}
\begin{proof}
\ref{lax1}: We obtain $\dom(\alpha(t))=\dom(\kappa(b_t) \cdot (b_t^{-1} \circ t)) = \df(t)$ from~\eqref{defa1}, Prop.~\ref{sane0} and~\ref{axpreimg0}.
\ref{lax2}: Choose $b$ with $\rng(t \circ \lambda) \subseteq \rng(t) \subseteq \rng(b)$.
Then $\alpha(t \circ \lambda) = \kappa(b) \cdot (b^{-1} \circ t \circ \lambda) = \kappa(b) \cdot (b^{-1} \circ t) \cdot \lambda = \alpha(t) \cdot \lambda$
by Prop.~\ref{welldef3} and~\ref{axsemi0}. 

\ref{lax3}$^{+}$: \textit{Special Case:} We assume $t$ is injective, $\df(t)=\{x_1,\dots,x_n\}$ and $\dom(v)=\{x_1,\dots,x_{n+k}\}$.
Proof by induction over k = $\#(\dom(v)\setminus\df(t))$.
\textit{Base Step:} We obtain $\alpha(t) = v \cdot \pi_{\df(t)} = v \cdot \pi_{\dom(v)} = v$ from the assumptions and~\ref{axneut0}.
\textit{Induction Step:} Set $u := v \cdot \pi_{\{x_1,\dots,x_{n+1}\}}$. By assumption, $t=:\langle x_1\mathord{:}t_1,\dots,x_n\mathord{:}t_n\rangle$
is injective, i.e. $t_1,\dots,t_n$ are pairwise distinct,
and we obtain $u \cdot \pi_{\{x_1,\dots,x_n\}} = v \cdot \pi_{\{x_1,\dots,x_{n+1}\}} \cdot \pi_{\{x_1,\dots,x_n\}}
= v \cdot \pi_{\df(t)} = \alpha(t)$ using~\ref{axsemi0}, and $\dom(u)=\{x_1,\dots,x_{n+1}\}$ using~\ref{axpreimg0}.
So $ut_1\dots t_n \in H$ by Prop.~\ref{hcon0}, and we obtain $\alpha(\eta_{ut_1\dots t_n}) = u = v \cdot \pi_{\{x_1,\dots,x_{n+1}\}}
= v \cdot \pi_{\df(\eta_{ut_1\dots t_n})}$ by Prop.~\ref{catch0} and the definition of $u$. Since $\eta_{ut_1\dots t_n}$ is again injective,
the induction hypothesis can be applied, i.e. we obtain $\tilde{t} \geq \eta_{ut_1\dots t_n} > t$ with $\alpha(\tilde{t}) = v$.

\textit{Extended Case:} Now we only assume $t$ is injective.
From $\alpha(t) = v \cdot \pi_{\df(t)}$ follows $v \leq \alpha(t)$ by~\ref{axproj0}, so $\df(t) \subseteq \dom(v)$
by Prop.~\ref{props1}~\ref{domain0}. Suppose $v=0$, then $\alpha(t)=0$ by~\ref{axzero0}, which means $\dom(\alpha(t))$ is infinite;
but $\dom(\alpha(t))=\df(t)$ by~\ref{lax1}, which is finite, contradiction! So $v \neq 0$. Thus, by Prop.~\ref{props1}\,\ref{finite0},
we have $\#\df(t) = n$ and $\#\dom(v)=n + k$ for some $n,k \in \mathbb{N}$.
Accordingly, let $\xi:\{x_1,\dots,x_n\} \rightarrow \df(t)$ be a bijection, and $\tilde{\xi}:\{x_1,\dots,x_{n+k}\} \rightarrow \dom(v)$
a bijective extension of $\xi$. Since $t$ is injective, $s:=t \circ \xi$ is injective; moreover $\df(s)=\{x_1,\dots,x_n\}$,
$\dom(v \cdot \tilde{\xi}) = \{x_1,\dots,x_{n+k}\}$ by~\ref{axpreimg0}, and
\begin{align*}
v \cdot \tilde{\xi} \cdot \pi_{\{x_1,\dots,x_n\}}
\underset{\ref{axsemi0}}{=} v \cdot \pi_{\df(t)} \cdot \xi
\underset{\text{Assmpt.}}{=} \alpha(t) \cdot \xi
\underset{\ref{lax2}}{=} \alpha(t \circ \xi)
= \alpha(s) \quad;
\end{align*}
so the Special Case can be applied, i.e. $s$ has an extension $\tilde{s}$ with $\alpha(\tilde{s})=v \cdot \tilde{\xi}$,
in particular $\df(\tilde{s})=\{x_1,\dots,x_{n+k}\}$ by~\ref{lax1}. Then $\tilde{t} := \tilde{s} \circ \tilde{\xi}^{-1}$ satisfies
\begin{align*}
\alpha(\tilde{t}) = \alpha(\tilde{s} \circ \tilde{\xi}^{-1})
\underset{\ref{lax2}}{=} \alpha(\tilde{s}) \cdot \tilde{\xi}^{-1}
= v \cdot \tilde{\xi} \cdot \tilde{\xi}^{-1}
\underset{\ref{axsemi0}}{=} v \cdot \pi_{\dom(v)}
\underset{\ref{axneut0}}{=} v \quad,
\end{align*}
and $\tilde{t}$ extends $t$, since $\tilde{t} \circ \pi_{\df(t)} = \tilde{s} \circ \tilde{\xi}^{-1} \circ \pi_{\df(t)}
= \tilde{s} \circ \pi_{\{x_1,\dots,x_n\}} \circ \xi^{-1} = s \circ \xi^{-1} = t \circ \xi \circ \xi^{-1} = t$.

\emph{General Case:} For arbitrary $t \in \mathrm{NTup}(H)$, the restriction $s := t|_{\rng(t^{-r})}$
(to the range of the right inverse $t^{-r}$) is injective. The function $\delta:\df(t) \rightarrow \df(s)$,
defined by $\delta := t^{-r} \circ t$, allows to recover $t$ from $s$; more specifically,
\begin{align}
s &= t \circ \pi_{\rng(\delta)} \quad,\\
t &= s \circ \delta \quad. \label{reconstruction1}
\end{align}
We have $v \leq v \cdot \pi_{\df(t)} = \alpha(t)$ by~\ref{axproj0} and the assumption, and moreover
\begin{align}
\alpha(t)
= \alpha(t \circ t^{-r} \circ t)
= \alpha(t \circ \delta)
\underset{\ref{lax2}}{=} \alpha(t) \cdot \delta
\underset{\text{Prop.}\,\ref{duplication0}}{\leq} e_\delta \quad,
\end{align}
so taken together $v \leq e_\delta$. The function $\tilde{\delta}$ with
\begin{align}
\label{deltaext0}
\tilde{\delta}(x) := \begin{cases}
\delta(x) \quad \text{ if } x \in \df(t) \\
x \qquad \text{ if } x \in \dom(v)\setminus\df(t)
\end{cases}
\end{align}
extends $\delta$ to $\dom(v)$, and by~\ref{axdom0} we have $v \leq d_{xx} = d_{x\tilde{\delta}(x)}$ for all $x \in \dom(v)\setminus\df(t)$,
so in fact $v \leq e_{\tilde{\delta}}$ (cf. Def.~\ref{ediag0}), which means
\begin{align}
\label{reconstruction0}
v = v \cdot \tilde{\delta}
\end{align}
by Prop.~\ref{duplication1}. Considering $\rng(\tilde{\delta}) \cap \df(s) = \df(s) = \df(t) \cap \rng(\delta)$ in the first equation below, we obtain
\begin{align*}
v \cdot \pi_{\rng(\tilde{\delta})} \cdot \pi_{\df(s)}
\underset{\ref{axsemi0}}{=} v \cdot \pi_{\df(t)} \cdot \pi_{\rng(\delta)}
&\underset{\text{Assmpt.}}{=} \alpha(t) \cdot \pi_{\rng(\delta)} \\
&\underset{\ref{lax2}}{=} \alpha(t \circ \pi_{\rng(\delta)})
= \alpha(s) \quad,
\end{align*}
and since $s$ is injective, by the Extended Case above, $s$ has an extension $\tilde{s}$ with
\begin{align}
\label{reduction0}
\alpha(\tilde{s}) = v \cdot \pi_{\rng(\tilde{\delta})} \quad.
\end{align}
Then $\tilde{t} := \tilde{s} \circ \tilde{\delta}$ satisfies $\alpha(\tilde{t})=v$, since
$\alpha(\tilde{t}) = \alpha(\tilde{s} \circ \tilde{\delta})
= \alpha(\tilde{s}) \cdot \tilde{\delta}
= v \cdot \pi_{\rng(\tilde{\delta})} \cdot \tilde{\delta}
= v \cdot \tilde{\delta}
= v$
by~\ref{lax2}, \eqref{reduction0}, \ref{axsemi0} and \eqref{reconstruction0}.
Moreover, we obtain
$\tilde{t} \circ \pi_{\df(t)}
= \tilde{s} \circ \tilde{\delta} \circ \pi_{\df(t)}
= \tilde{s} \circ \delta
= s \circ \delta
= t$
using~\eqref{deltaext0} and~\eqref{reconstruction1},
which means that $\tilde{t}$ extends $t$. This concludes the proof of~\ref{lax3}$^{+}$.
\qed
\end{proof}
\begin{proposition}
\label{quasi2}
$\alpha:\mathrm{NTup}(H) \rightarrow V$ is a quasi-labeling with $\rng(\alpha) = V \setminus \{0\}$.
\end{proposition}
\begin{proof}
Proposition~\ref{lplus0} shows \ref{lax1}, \ref{lax2} and~\ref{lax3}$^{+}$. For $v \in V\setminus\{0\}$,
we obtain $\alpha(\langle\rangle) = \kappa(\langle\rangle) \cdot \pi_{\emptyset} = 1 \cdot \pi_{\emptyset} = 1 = v \cdot \pi_{\emptyset}$
from Prop.~\ref{welldef2} and~\ref{axone0}, so by~\ref{lax3}$^{+}$ there exists $\tilde{t} \geq \langle\rangle$ with $\alpha(\tilde{t}) = v$.
On the other hand, $\dom(\alpha(t))=\df(t)$ by~\ref{lax1}, which is finite, so $\alpha(t) \neq 0$. This shows $\rng(\alpha)=V\setminus\{0\}$.

\ref{lax3}: Assume $\alpha(t) \leq v \cdot \pi_{\df(t)}$ and $\df(t) \subseteq \dom(v)$. Set $w:=v \wedge \alpha(t)$, then
$w \cdot \pi_{\df(t)} = (v \wedge \alpha(t)) \cdot \pi_{\df(t)} = v \cdot \pi_{\df(t)} \wedge \alpha(t) = \alpha(t)$ by~\ref{axdist0},
\ref{lax1}, \ref{axneut0} and the assumption.
So by~\ref{lax3}$^{+}$ there exists $\tilde{t} \geq t$ with $\alpha(\tilde{t}) = w \leq v$.
Also $t \in H^{\dom(v)}$, since $\df(\tilde{t}) = \dom(w) = \dom(v) \wedge \dom(\alpha(t)) = \dom(v) \cup \df(t) = \dom(v)$
by~\ref{lax1}, Prop.~\ref{props1}\,\ref{dominf0} and the assumption $\df(t) \subseteq \dom(v)$.
\qed
\end{proof}

\subsection{Representation Theorem}
\begin{theorem}[Representation Theorem]
Every orbital semilattice $V$ can be embedded in an orbital table semilattice $\mathrm{Tab}(G)$.
\end{theorem}
\begin{proof}
The quasi-labeling $\alpha$ of Sect.~\ref{construction1} satisfies $\rng(\alpha) = V \setminus \{0\}$ (cf. Prop.~\ref{quasi2}).
The tuple labeling $\bar{\alpha}$ of Thm.~\ref{labeling0} then also satisfies $\rng(\bar{\alpha}) = V \setminus \{0\}$,
as is evident from~\eqref{quotient0}. So by Thm.~\ref{embedding0}, $\ext_{\bar{\alpha}}$ is the claimed embedding.
\qed
\end{proof}

\bibliographystyle{splncs03}
\bibliography{arxiv}
\end{document}